\documentclass[12pt]{article}
\usepackage{amsfonts,amssymb,color,url}
\usepackage{longtable}
\usepackage[all,knot,arc]{xy}
\usepackage{cite}

\newtheorem{theorem}{Theorem}[section]
\newtheorem{prop}[theorem]{Proposition}
\newtheorem{lemma}[theorem]{Lemma}
\newtheorem{cor}[theorem]{Corollary}
\newtheorem{unnumber}{}

\newtheorem{prepf}[unnumber]{Proof}
\newenvironment{proof}{\prepf\rm}{\endprepf}
\newcommand{\qed}{\hfill$\Box$}

\newtheorem{preprob}{Problem}
\newenvironment{problem}{\preprob\rm}{\endpreprob}

\newcommand{\agl}{\mathop{\mathrm{AGL}}}

\newcommand{\agaml}{\mathop{\mathrm{A}\Gamma\mathrm{L}}}

\newcommand{\Sym}{\mathop{\mathrm{Sym}}\nolimits}
\newcommand{\sym}{S_n}
\newcommand{\alt}{A_n}
\newcommand{\trans}{T_n}

\newcommand{\End}{\mathop{\mathrm{End}}}

\begin{document}

\title{Primitive Permutation Groups and Strongly Factorizable Transformation Semigroups}
\author{Jo\~ao Ara\'ujo\footnote{Departamento de Matem\'atica and CMA, Faculdade de Ci\^encias e Tecnologia (FCT)  
Universidade Nova de Lisboa (UNL), 2829-516 Caparica, Portugal;   jj.araujo@fct.unl.pt} 
 \footnote{CEMAT-Ci\^{e}ncias, 
Faculdade de Ci\^{e}ncias, Universidade de Lisboa
1749--016, Lisboa, Portugal; jjaraujo@fc.ul.pt} , Wolfram Bentz\footnote{School of Mathematics \& Physical Sciences, University of Hull, Kingston upon Hull, HU6 7RX, UK; W.Bentz@hull.ac.uk}, and Peter J. Cameron\footnote{School of Mathematics and Statistics, University of St Andrews, North Haugh, St Andrews, Fife KY16 9SS, UK; pjc20@st-andrews.ac.uk}}
\date{}
\maketitle

\begin{abstract}
Let $\Omega$ be a finite set and $T(\Omega)$ be the full transformation monoid on $\Omega$. The rank of a transformation $t\in T(\Omega)$ is the natural number $|\Omega t|$.  Given $A\subseteq T(\Omega)$, denote by $\langle A\rangle$ the semigroup generated by $A$.  Let $k$ be a fixed natural number such that $2\le k\le |\Omega|$. In the first part of this paper we (almost) classify the permutation groups $G$ on $\Omega$ such that  for all rank $k$ transformation $t\in T(\Omega)$, every element in  $S_t:=\langle G,t\rangle$ can be written as a product $eg$, where $e^2=e\in S_t$ and $g\in G$.  In the second part we prove, among other results, that if $S\le T(\Omega)$ and $G$ is the normalizer of $S$ in the symmetric group on $\Omega$, then the semigroup $SG$ is regular if and only if $S$ is regular. (Recall that a semigroup $S$ is regular if for all $s\in S$ there exists $s'\in S$ such that $s=ss's$.) The paper ends with a list of problems.  

\vspace{2cm}
\end{abstract} 

\section{Introduction}

A semigroup $S$ with set of idempotents $E$ and group of units $G$ is said to be {\em strongly factorizable} if $S=EG$. 
Denote by $\sym$ the symmetric group  on a set $\Omega$ of cardinality $n$ and denote by $\trans$ the full transformation semigroup on the same set. It is clear that $\sym$ is the group of units of $\trans$.  Given $t\in \trans$ the {\em rank} of $t$ is the cardinality of the set $\Omega t$. 

The first goal of this paper is to prove the following sequence of theorems. 
We note that, for $1<k<n-1$, the $k$-transitive or $k$-homogeneous groups are
explicitly known.

\begin{theorem}\label{rank2}
Let $n\ge2$, and $G\le \sym$ be a group. The following are equivalent: 
\begin{enumerate}
\item for all rank $2$ transformations $t\in \trans$ the semigroup $\langle G,t\rangle$ is strongly factorizable; 
\item $G$ is primitive.  
\end{enumerate}
\end{theorem}

\begin{theorem}\label{rankn}
Let $G\le \sym$ and let $k$ be a fixed natural number such that $6\le k\le n$.   The following are equivalent: 
\begin{enumerate}
\item for all rank $k$ transformations $t\in \trans$ the semigroup $\langle G,t\rangle$ is strongly factorizable; 
\item $G=\sym$ or  $G=\alt$ (with $n\ne k$) in their natural action on $n$ points.
\end{enumerate}
\end{theorem}

\begin{theorem}\label{rank5}
Let $n\ge 5$, and $G\le \sym$. 
The following are equivalent: 
\begin{enumerate}
\item for all rank $5$ transformations $t\in \trans$ the semigroup $\langle G,t\rangle$ is strongly factorizable; 
\item $G$ is $5$-transitive or $G=A_6$ ($n=6$). 
\end{enumerate}
\end{theorem}

\begin{theorem}\label{rank4}
Let $n\ge 4$, and $G\le \sym$. 
If $G$ is $4$-transitive, or $G$ is one of $A_5$ ($n=5$),  or $M_{11}$ ($n=12$), then 
  the semigroup $\langle G,t\rangle$ is strongly factorizable, for all rank $4$ transformations $t\in \trans$. 

Any other group satisfying this property  
contains $\mathrm{PSL}(2,2^p)$ ($n=2^p+1$), where $2^p-1$ is a (Mersenne) prime.
\end{theorem}

\begin{theorem}\label{rank3}
Let $n \ge 3$, $G\le \sym$. 
If $G$ satisfies one of the following properties:
\begin{enumerate}
\item $G$ is $3$-transitive;
\item $\mathrm{PSL}(2,q)\le G \le \mathrm{P}\Gamma \mathrm{L}(2,q)$ where $q \equiv 3~\mbox{ mod }~4$ ($n=q+1$); 
\item $\mathrm{PSL}(2,q)\le G \le \mathrm{P}\Sigma \mathrm{L}(2,q)$ where $q \equiv 1~\mbox{ mod }~4$ ($n=q+1$); 
\item $G=\mathrm{Sp}(2d,2)$, $d \ge 3$, in either of its $2$-transitive representations ($n=2^{2d-1}\pm 2^{d-1}$);
\item $G=2^d: \mathrm{Sp}(d,2)$, $d \ge4$ and even ($n=2^d$);
\item $G$ is one of $A_4$ ($n=4$), $\mathrm{PSL}(2,11)$ ($n=11$), $2^4:A_6$ ($n=16$), $2^6:G_2(2)$ or its subgroup of index $2$ ($n=64$), Higman--Sims ($n=176$), $\mathrm{Co}_3$ ($n=276$);
\end{enumerate}
then the semigroup  $\langle G,t\rangle$ is strongly factorizable, for all rank $3$ transformations $t\in \trans$. 

Any other group satisfying the previous property must satisfy
$\agl(1,2^p)\le G\le\agaml(1,2^p)$ with $p$ and $2^p-1$ prime.
\end{theorem}

These results required the repeated use of the classification of finite simple groups.  This completes the  part of the paper dealing with semigroups $S$ in which $S=EG$. 

In the second part we turn to semigroups of the form $SG$, where $S\le \trans$ and $G$ is the normalizer of $S$ in $\sym$. The general goal is to decide if some properties of $SG$ carry to $S$. The main theorem is the following.

\begin{theorem}
Let $S\le T_n$ be a transformation monoid and $G$ be the normalizer of $S$ in the symmetric group $\sym$. 
Then  $SG$ is regular if and only if  $S$ is regular.
\end{theorem}
As often happens in semigroup theory, the proof is short but  tricky. 

Finally we prove the following result that generalizes a well known theorem in groups. 

\begin{theorem}
Let $S$ be a finite semigroup. Then every element of $S$ has a square root if and only if  every element in $S$ belongs to a subgroup of odd order. 
\end{theorem}

A semigroup $S$ is said to be {\em factorizable} if there exist two sets $A,B\subset S$ such that $S=AB$. Factorizable semigroups were prompted by the study of extensions (direct products, Zappa--Sz\'ep and Catino extensions \cite{catino}), but the study went far beyond that initial motivation leading to many papers in very different contexts (structural semigroup theory, topological semigroups, presentations, automata and languages, combinatorial semigroup theory, Morita equivalences, etc.). If  $S=EG$, where $E$ is its set of idempotents and $G$ is its group of units, then $S$ is said to be {strongly factorizable}.
The class of strongly factorizable semigroups contains the two most studied transformation semigroups: $\trans$, the full transformation monoid on a finite set  and the monoid of endomorphisms of a finite-dimensional vector space. As, by an argument similar to the one used to prove Cayley's Theorem in groups, every finite semigroup embeds in some finite $\trans$, it follows that every finite semigroup is a subsemigroup of a strongly factorizable semigroup. Recall that a semigroup $S$ is said to be {\em regular} if for all $a\in S$ there exists $b\in S$ (called an {\em inverse} of $a$) such that $a=aba$ and $b=bab$. If every element in a regular semigroup has only one inverse, then the semigroup is said to be {\em inverse}.  Not every strongly factorizable semigroup is inverse (consider a non-commutative semigroup of idempotents with identity), but they are all regular semigroups. In fact, if $eg\in EG$ then $eg=(eg)(g^{-1}e)(eg)$ and $g^{-1}e=(g^{-1}e)(eg)(g^{-1}e)$.

Finally we note that in the context of inverse semigroups (where strongly factorizable semigroups are simply called factorizable) this topic is an entire body of knowledge on its own. (For details  we suggest the lovely survey paper \cite{FitzGerald2009} and the references therein.)  In the context of groups, it is usually assumed that at least one of $A$ and $B$ is a subgroup; many classifications are known, notably that of Liebeck, Praeger and Saxl~\cite{lps,lps2}, who found
all factorizations of almost simple groups where both factors are subgroups.

Throughout the twentieth century there was a doctrine which stated that a problem in semigroups was considered solved when reduced to a question in groups. This dramatically changed in this century when it was realized that it would be much more productive for both sides to keep an ongoing conversation. One of the driving forces of this conversation has been the following general problem, whose study has led to significant research on permutation groups:
\begin{quote}
Classify the pairs $(G,a)$, where $a$ is a map on a finite set $\Omega$ and $G$ is a group of permutations of $\Omega$,
such that the semigroup  $\langle G,a\rangle$ generated by $a$ and $G$  has a given property $P$.
\end{quote}

A very important class of groups that falls under this general scheme is that of \emph{synchronizing groups}, groups of permutations on a set that together with any non-invertible map on the same set generate a constant (see \cite{steinberg,abc,abcrs,arcameron22,acs,cameron,neumann}). These groups are very interesting from a group theoretic  point of view and are linked  to the \emph{\v{C}ern\'y conjecture},  a longstanding open problem in automata theory.

Three other sample instances of the general question are the following: 
\begin{enumerate}
\item Let $A\subseteq \trans$; classify the permutation groups $G\le \sym$ such that $\langle G,a\rangle$ is regular for all $a\in A$. For many different sets $A$, this problem has been considered in    \cite{abc2,abc3,ac,AMS,lmm,lm,levi96,mcalister}, among others. 
\item Classify the permutation groups $G\le \sym$ such that  for all $a\in\trans\setminus \sym$ the equality $\langle G,a\rangle\setminus G=\langle \sym,a\rangle\setminus \sym$ holds. This problem was solved in \cite{AAC}.
\item Classify the permutation groups $G\le \sym$ such that for all $a\in \trans\setminus \sym$ we have $\langle G,a\rangle \setminus \sym = \langle g^{-1}ag\mid g\in G\rangle$. This classification \cite{acmn} answered an old problem.
\end{enumerate}

We saw above the importance of strongly factorizable semigroups and the multitude of contexts in which they appear. The first goal of this paper is to continue the trend described above and classify the permutation groups that together with any transformation of a given rank $k$ generate a strongly factorizable semigroup. 

The second part of the paper deals with the following problem proposed by the third author. Observe that in the founding paper of factorizable semigroups \cite{tolo} the goal was to check when some given properties of $A$ and $B$ carry to the factorizable oversemigroup $S:=AB$. Here we go the converse direction: given $T=SG$, where $S\le \trans$ and $G$ is the normalizer of $S$ in $\sym$, find semigroup properties that carry from $SG$ to $S$. This looks  a sensible question since in $SG$ we can take advantage of the group theory machinery and hence checking a property might be easier in $SG$ than in $S$.

The main result of this part of the paper says that {\em regularity} carries from $SG$ to $S$.  

We now summarise the contents of the paper.
In Section \ref{two}  we classify permutation groups with what we call the {\em ordered $k$-ut property}. This is the cornerstone of our classification results. Section \ref{three} connects the group theory results of the previous section with the theorems on factorizable semigroups.    Section \ref{four} deals with the semigroups $SG$, where $S\le \trans$ and $G$ is its normalizer in $\sym$. The paper finishes with a list of open problems. 

\section{The ordered $k$-ut property}\label{two}

A permutation group $G$ on $\Omega$ is said to have the \emph{$k$-universal
transversal property} (or $k$-ut for short) if, given any $k$-subset $A$
and $k$-partition $P$ of $\Omega$, there exists $g\in G$ such that $Ag$ is a
transversal to $P$. These groups were studied in connection with permutation
groups $G$ such that $\langle G,a\rangle$ is regular for all maps $a$ of
rank $k$. The groups satisfying the $k$-ut property for $3\le k\le n/2$ were partly classified in\cite{ac} (small corrections to the case of $3$-ut where made in \cite{abc}).

A permutation group $G$ on $\Omega$ is said to have the \emph{ordered $k$-ut
property} if, given any ordered $k$-subset $A=(a_1,\ldots,a_k)$ and ordered
$k$-partition $\pi=(P_1,\ldots,P_k)$ of $\Omega$, there exists $g\in G$ such
that $a_ig\in P_i$ for $i=1,\ldots,k$. 

Our goal is to classify the groups possessing ordered $k$-ut. 
Clearly, ordered $k$-ut implies the usual $k$-ut, so we only need look among
permutation groups with $k$-ut.

\subsection{Permutation group properties}

A permutation group is \emph{$k$-primitive} if it is $k$-transitive and the
stabiliser of $k-1$ points is primitive on the remaining points.

A permutation group is \emph{generously $k$-transitive} if, given any 
$(k+1)$-set $M\subseteq\Omega$, the group induced on $M$ by its setwise
stabiliser is the symmetric group of degree $k+1$. It is straightforward to
prove that a generously $k$-transitive group is indeed $k$-transitive.

The next result summarises the relationship between these concepts and
the ordered $k$-ut property.

\begin{prop}\label{p:basicprop}
\begin{enumerate}
\item
A $k$-transitive group has the ordered $k$-ut property.
\item
For $k\ge 2$, ordered $k$-ut implies ordered $(k-1)$-ut.
\item 
A permutation group $G$ which has $k$-ut and is generously $(k-1)$-transitive
has ordered $k$-ut.
\item
A permutation group $G$ with the ordered $k$-ut property is $(k-1)$-transitive.
\item
A permutation group $G$ with the ordered $k$-ut property is $(k-1)$-primitive.\label{list:primitive}
\item
For $k\ge 2$, if $G$ has ordered $k$-ut, its point stabiliser has ordered $(k-1)$-ut.
\end{enumerate}
\end{prop}

\begin{proof}
(a), (b) and (f) are straightforward.

(c) Given a $k$-set $A$ and a $k$-partition $\pi$, there is an element of $G$ mapping $A$ to a
transversal of $\pi$; premultiplying this element by an element in the
setwise stabiliser of $A$ shows that we can map elements of $A$ to parts
of $\pi$ in any order.

(d) Let $(a_1,\ldots,a_{k-1})$
and $(b_1,\ldots,b_{k-1})$ be two ordered $(k-1)$-tuples of distinct points
of $\Omega$. If $x$ is any point different from $a_1,\ldots,a_{k-1}$, then
a permutation mapping $a_1,\ldots,a_{k-1},x$ to a transversal of the partition
$\{b_1\},\ldots,$ $\{b_{k-1}\},\Omega\setminus\{b_1,\ldots,b_{k-1}\}$ maps the
first $(k-1)$-tuple to the second.

(e) Suppose that $G$ is not $(k-1)$-primitive; let $B$ be a
non-trivial block of imprimitivity for the stabiliser of distinct $a_1,\ldots,a_{k-2}\in \Omega$.
Let $A$ be a subset consisting of $a_1,\ldots,a_{k-2}$ and two points $b_1,b_2$ of $B$,
and $P$ the partition into $\{a_1\}$, \dots, $\{a_{k-2}\}$, $B$, and the
rest of $\Omega$. Any permutation mapping $a_i$ to $a_i$ for
$i=1,\ldots,k-2$, maps $b_1,b_2$ either both into $B$ or outside of $B$. Hence $G$ does not have the ordered $k$-ut property.\qed
\end{proof}

\begin{prop}\label{p:ord2ut}
A permutation group $G$ has the ordered $2$-ut property if and only if it is
primitive.
\end{prop}

\begin{proof} Ordered $2$-ut implies primitivity, by (\ref{list:primitive}) above.
Conversely, suppose $G$ is primitive. Then all
orbital digraphs for $G$ are connected, and hence (since $G$ is transitive)
strongly connected. Now let $A=\{a_1,a_2\}$ be a $2$-set and
$\pi=\{P_1,P_2\}$ a $2$-partition. Since the orbital graph with edge set
$(a_1,a_2)^G$ is strongly connected, there is an edge with initial vertex
in $P_1$ and terminal vertex in $P_2$; the element of $G$ mapping $(a_1,a_2)$
to this edge witnesses ordered $k$-ut.
\qed\end{proof}

The next proposition gives sufficient conditions for generous $k$-transitivity.

\begin{prop}
\begin{enumerate}
\item
Suppose that $G$ is $k$-transitive, and every orbital of the $(k-1)$-point
stabiliser is self-paired. Then $G$ is generously $k$-transitive.
\item
Suppose that $G$ is $k$-transitive, and the non-trivial orbits of the
stabiliser of $k$ points all have different sizes. Then $G$ is generously
$k$-transitive.
\end{enumerate}
\end{prop}

\begin{proof}
(a)
Take any $k+1$ points $a_1,\ldots,a_{k+1}$. By assumption, $G$ has an
element fixing $a_1,\ldots,a_{k-1}$ and interchanging $a_k$ with $a_{k+1}$.
Since the numbering of the points is arbitrary, the setwise stabiliser of the
set of $k+1$ points induces every possible transposition on it. The
transpositions generate the symmetric group.

(b)
This follows immediately from (a), since paired orbits have the same sizes.
\qed\end{proof}

\subsection{The classification of the groups with the ordered $k$-ut property}

Trivially, $G \le S_n$ has the ordered $n$-ut property if and only if $G=S_n$.
Any permutation group has the ordered $1$-ut property and by Proposition \ref{p:ord2ut}, ordered $2$-ut is equivalent to primitivity. 

Ordered $k$-ut clearly implies $k$-ut, and hence (by
Proposition~\ref{p:basicprop}(b)), $k'$-ut for all $1\le k'\le k$. 
Hence it remains to consider the groups arising in the classification of
groups with $k$-ut from \cite{abc,ac} (given below), except that 
these results only classify the values of $n$ with $ \lfloor \frac{n+1}{2}\rfloor \ge k$. Below, we will deal with smaller values of $n$ by ad-hoc arguments. 

\begin{prop} \label{p:ord6ut} Let $n\ge 6$, $G\le S_n$, then $G$ has the ordered $k$-ut property  for some $6\le k \le n$, if and only if  $G=S_n$ or $G=A_n$ (with $n\ne k$) in their natural action on $n$ points.
\end{prop}

\begin{proof} By \cite[Theorem 1.4]{ac}, for $n\ge 11$, the only groups with $6$-ut are $A_n$ and $S_n$, hence no other group has ordered~$6$-ut, and so it does not have ordered $k$-ut either. 
For  $n\le 10$, the listed groups are the only ones that are $(k-1)$-transitive.

Conversely, it is easy to check that the listed values of  $A_n$ and $S_n$ have the ordered~$k$-ut property, for $6 \le k\le n$.
\qed\end{proof}

\begin{prop} \label{p:ord5ut}Let $n\ge 5$, $G\le S_n$, then $G$ has the ordered $5$-ut property if and only if it is $5$-transitive or $A_6$ ($n=6$).
\end{prop}

\begin{proof}
For $n\ge 11$, a group with $5$-ut is $5$-homogeneous or
$\mathrm{P}\Gamma\mathrm{L}(2,32)$ (with degree $33$) \cite[Theorem 1.5]{ac}. The $5$-homogeneous groups (with $n\ge 10$) are $5$-transitive
and have ordered $5$-ut, while $\mathrm{P}\Gamma\mathrm{L}(2,32)$  is not $4$-transitive
so does not have ordered $5$-ut.

For $n\le 10$, $A_6$, which clearly satisfies ordered~$5$-ut, is the only
group that is $4$-transitive, but not $5$-transitive. 
\qed\end{proof}

\begin{prop} \label{p:ord4ut} Let $n\ge 4$, $G\le S_n$, then $G$ has the ordered $4$-ut property if it is $4$-transitive, $A_5$ ($n=5$), or $M_{11}$ ($n=12$). If there are any other groups with ordered $4$-ut, they 
contain $\mathrm{PSL}(2,2^p)$ ($n=2^p+1$), where $2^p-1$ is a (Mersenne) prime.
\end{prop}

\begin{proof}
By \cite[Theorems 1.3, 1.6]{ac} for $n\ge8$, a group with $4$-ut is $4$-homo\-geneous or
$M_{11}$ ($n=12$), or possibly almost simple with socle $\mathrm{PSL}(2,q)$ where $q$ is
prime or $q=2^p$ for some prime $p$ (with $n=q+1$). The $4$-homogeneous groups with $n\ge 8$ are $4$-transitive except for $\mathrm{PSL}(2,8)$, $\mathrm{P\Gamma L}(2,8)$ ($n=9$), and 
$\mathrm{P\Gamma L}(2,32)$ ($n=33$). 

For $4\le n\le7$, the only $3$-, but not $4$-transitive groups are  $A_5$ ($n=5$) and $\mathrm{PGL}(2,5)$ ($n=6$).

The $4$-transitive groups have ordered $4$-ut, and
the Mathieu group $M_{11}$ ($n=12$) is generously
$3$-transitive (the orbit lengths for the $3$-point stabiliser are $3$ and
$6$), and thus also has ordered $4$-ut.  Almost simple groups contained in $\mathrm{PSL}(2,q)$ with $q\ge5$ prime are not
$3$-primitive (the stabiliser of two points has a normal cyclic subgroup of
composite order $q-1$).

Now consider $G=\mathrm{P}\Gamma\mathrm{L}(2,2^p)$, for $2^p-1$ composite. Once again, these groups have the property that the
$2$-point stabiliser has a regular normal subgroup which is cyclic of 
composite order; so they are not $3$-primitive. 

Finally, $A_5$ clearly has ordered $4$-ut.
\qed\end{proof}

We remark that computation shows that the groups $\mathrm{PSL}(2,8)$, $\mathrm{P}\Gamma\mathrm{L}(2,8)$ ($n=9$), and $\mathrm{P}\Gamma\mathrm{L}(2,32)$ ($n=33$) satisfy
ordered $4$-ut. 

\medskip

Before we consider the case $k=3$, we will give an updated list of the status of the $3$-ut property. The following theorem combines results from \cite{ac} with the corrections from \cite{abc} and adds the (trivial) cases with $n=3,4$.

\begin{prop}\label{p:3ut}  Let $n\ge 3$, $G\le S_n$, then $G$ has the $3$-ut property if it satisfies one of the following properties:
\begin{enumerate}
\item $G$ is $3$-homogeneous;
\item $\mathrm{PSL}(2,q)\le G \le \mathrm{P}\Sigma \mathrm{L}(2,q)$ where $q \equiv 1 \mbox{ mod } 4$ ($n=q+1$); 
\item $G=\mathrm{Sp}(2d,2)$, $d \ge 3$, in either of its $2$-transitive representations ($n=2^{2d-1}\pm 2^{d-1}$);
\item $G=2^d: \mathrm{Sp}(d,2)$, $d \ge4$ and even ($n=2^d$);
\item $G$ is one of $\mathrm{AGL}(1,7)$, ($n=7$), $\mathrm{PSL}(2,11)$ ($n=11$), $2^4:A_6$ ($n=16$), $2^6:G_2(2)$ or its subgroup of index $2$ ($n=64$), $\mathrm{Sz}(8)$, $\mathrm{Sz}(8):3$ ($n=65$), Higman-Sims ($n=176$), $\mathrm{Co}_3$ ($n=276$);

\end{enumerate}

If there are any other groups with $3$-ut, they are one of the following:
\begin{enumerate}
\item[(e)] Suzuki groups $\mathrm{Sz}(q)$ with $q\ge 32$, potentially extended by field automorphisms ($n=q^2+1$);
\item[(f)] $\mathrm{AGL}(1,q)\le G\le \mathrm{A}\Gamma\mathrm{L}(1,q)$, where $q$ is either prime with $q \equiv 11 \mbox{ mod } 12$, or $q=2^p$ with $p$ prime, and  for all $c \in \mathrm{GF}(q)\setminus\{0,1\}$, 
$|\langle -1,c,c-1\rangle|=q-1$  $(n=q$);
\item[(g)]  subgroups of index $2$ in $\mathrm{AGL}(1,q)$, with $q \equiv 11~\mbox{ mod }~12$ and prime, and  for all $c \in \mathrm{GF}(q)\setminus\{0,1\}$, 
$|\langle -1,c,c-1\rangle|=q-1$  $(n=q$).
\end{enumerate}
\end{prop}
%

With this result, we can prove that the groups with ordered $3$-ut are just
those listed in Theorem~\ref{rank3}.

\begin{prop}\label{p:ord3ut}
Let $n \ge 3$, $G\le \sym$. 
If $G$ satisfies one of the following properties:
\begin{enumerate}
\item $G$ is $3$-transitive;
\item $\mathrm{PSL}(2,q)\le G \le \mathrm{P}\Gamma \mathrm{L}(2,q)$ where $q \equiv 3~\mbox{ mod }~4$ ($n=q+1$); 
\item $\mathrm{PSL}(2,q)\le G \le \mathrm{P}\Sigma \mathrm{L}(2,q)$ where $q \equiv 1~\mbox{ mod }~4$ ($n=q+1$); 
\item $G=\mathrm{Sp}(2d,2)$, $d \ge 3$, in either of its $2$-transitive representations ($n=2^{2d-1}\pm 2^{d-1}$);
\item $G=2^d: \mathrm{Sp}(d,2)$, $d \ge4$ and even ($n=2^d$);
\item $G$ is one of $A_4$ ($n=4$), $\mathrm{PSL}(2,11)$ ($n=11$), $2^4:A_6$ ($n=16$), $2^6:G_2(2)$ or its subgroup of index $2$ ($n=64$), Higman--Sims ($n=176$), $\mathrm{Co}_3$ ($n=276$);
\end{enumerate}
then it satisfies ordered $3$-ut.

Any other group satisfying ordered $3$-ut must satisfy
$\agl(1,2^p)\le G\le\agaml(1,2^p)$ with $p$ and $2^p-1$ prime.
\end{prop}

\begin{proof}
According to  Proposition \ref{p:3ut}, groups with
$3$-ut are one of the five types in the proposition, or potentially one of the three additional types listed.

If $G$ is $3$-transitive, then it has the ordered $3$-ut property.

Suppose that $G$ is $3$-homogeneous but not $3$-transitive. 
If $3\le n \le 5$, then $G$ is $A_4$ ($n=4$), which clearly has ordered $3$-ut, $\agl(1,5)$ ($n=5$), which we will exclude below, or nor $2$-transitive and hence 
does not have ordered $3$-ut.

If $n\ge 6$, then $G$ is $\agl(1,8)$, $\agaml(1,8)$ ($n=8$) or $\agaml(1,32)$ ($n=32$), or $G$ contains $\mathrm{PSL}(2,q)$ for
$q\equiv3$~(mod~$4$) ($n=q+1$). The affine groups are included in our undecided cases, so assume $G$ contains such a $\mathrm{PSL}(2,q)$.

Let $\{a_1,a_2,a_3\}$ be an ordered $3$-set and $\{P_1,P_2,P_3\}$ an ordered $3$-partition of the underlying set $\Omega$.
Without loss of generality, we may assume that the parts $P_1,P_2,P_3$ are arranged in
increasing order of size.
Using the transitivity of $G$, we may assume that $a_1$ can be mapped
into $P_1$, and indeed that $a_1\in P_1$. Now the set
\[\{(x,y):(a_1,x,y)\in(a_1,a_2,a_3)^G\}\]
is the edge set of a Paley tournament on $\Omega\setminus\{a_1\}$. If this
tournament includes an arc from $P_2$ to $P_3$, then we are done; so suppose
not. If $|P_2|=1$, then $|P_1|=1$, and all arcs between $P_2$ and $P_3$ point
into the point in $P_2$; so the tournament has out-degree at most $1$, a
contradiction. So suppose that $|P_2|>1$.
In the Paley tournament on $q$ points, any two points are dominated by
precisely $(q-3)/4$ points; but if there are no arcs from $P_2$ to $P_3$, then
two points in $P_2$ are dominated by every point in $P_3$, and by assumption
there are at least $q/3$ such points. So $(q-3)/4\ge q/3$, a contradiction.

So $\mathrm{PSL}(2,q)$, and any overgroup, has ordered $3$-ut for $q\equiv3$~(mod~$4$).

We claim that, with the exceptions of $\mathrm{AGL}(1,7)$ ($n=7$), $\mathrm{Sz}(8)$ and $\mathrm{Sz}(8):3$ ($n=65$),  types (b)--(e) in Proposition \ref{p:3ut} are generously $2$-transitive, and so have
ordered $3$-ut.  For types (c),(d), and most groups of type (e), the $2$-point stabilisers have
all orbits of different sizes, these being
\begin{itemize}\itemsep0pt
\item $2^{2d-2}\pm2^{d-1}-2$ and
$2^{2d-2}$ for type (c); 
\item $2^{2d-1}-2$ and $2^{2d-1}$ for type (d); 
\item $3$, $6$ for $\mathrm{PSL}(2,11)$ ($n=11$);
\item $6$, $8$ for $2^4:A_6$ ($n=16$);
\item $6$, $24$ and $32$ for $2^6:G_2(2)$ and its subgroup ($n=64$); 
\item $12$, $72$ 
and $90$ for $\mathrm{HS}$ ($n=176$); 
\item $112$ and $162$ for $\mathrm{Co}_3$ ($n=276$). 
\end{itemize} 

For (b), since the point stabiliser has
even order and rank at most $3$, all its non-trivial orbitals are self-paired.

For the groups of type (f), if $q-1$ is not prime, then the point stabiliser of $\mathrm{A}\Gamma\mathrm{L}(1,q)$ has a proper normal subgroup, and
so these groups are not $2$-primitive, and hence do not have ordered $3$-ut. The same argument excludes $\agl(1,5)$ and $\agl(1,7)$.

Of the remaining groups, we observe that
the Suzuki groups do not have ordered $3$-ut since they are not $2$-primitive
(the point stabiliser has a normal subgroup of order $q$); and subgroups of
index $2$ in $\mathrm{AGL}(2,q)$ for odd $q$ fail to be $2$-transitive. 

Hence the only open cases remaining are groups containing $\agl(1,q)$ with $q-1$ is prime, which occurs only if $q=2^p$ for $p$ prime (and 
$2^p-1$ is a Mersenne prime). 
\qed\end{proof}
Computation shows that $\mathrm{AGL}(1,32)$ and
$\mathrm{A}\Gamma\mathrm{L}(1,32)$ do indeed have ordered $3$-ut. 

\section{Strongly factorizable semigroups and maps with fixed rank}\label{three}

A monoid $S$ with group of units $G$ and set of idempotents $E$ is said to be {\em strongly factorizable} if $S=EG$. 
Let $\Omega$ be a finite set. Every finite semigroup can be embedded in some $T(\Omega)$, a strongly factorizable semigroup. More generally, any semigroup $S$ such that $\Sym(\Omega)\le S\le T(\Omega)$ is strongly factorizable. 

Let $k$ be a natural number; the goal of this section is to classify the groups $G\le \Sym(\Omega)$ such that $\langle G,t\rangle$ is strongly factorizable for all rank $k$ transformation $t\in T(\Omega)$. The next result links this goal and the results of the previous sections. 


\begin{lemma}
Let $\Omega$ be a finite set, $G\le\Sym(\Omega)$  and $k\le|\Omega|$. Then the following are equivalent:
\begin{enumerate}
\item $G$ possesses the ordered $k$-ut property;
\item for all rank $k$ transformations $t\in T(\Omega)$, we have 
\[
\langle G, t\rangle = EG,
\] where $E$ is the set of idempotents of $\langle G ,t\rangle$. 
\end{enumerate} 
\end{lemma}
\begin{proof}

First, assume that $G$ has the ordered $k$-ut property.
Let $a\in\langle G,t\rangle$ be a map of rank $l\le k$. As $G$ has ordered $k$-ut, it has ordered $l$-ut (Proposition \ref{p:basicprop} (b)). Therefore, given a sequence of kernel classes of $a$, say   $(A_1,\ldots,A_l)$, and the corresponding $l$-tuple of images $(A_1 a,\ldots,A_l a)$, there exists $g\in G$ such that $A_i a g \subseteq A_i$, for all $i\in \{1,\ldots,l\}$; thus $ag$ is an idempotent and $a=(ag)g^{-1}\in EG$.  This proves the direct implication. 

Conversely, let $(A_1,\ldots,A_k)$ be a $k$-partition of $\Omega$ and let $(a_1,\ldots ,a_k)$ be a $k$-tuple of different elements of $\Omega$. We claim that there exists $g\in G$ such that $a_ig \in A_i$, for all $i\in \{1,\ldots,k\}$. In fact, let $t\in T(\Omega)$ be a map such that $A_i t =\{a_i\}$.  By assumption $\langle G,t\rangle$ is  strongly factorizable and hence $t=eg$, for some $g\in G$ and idempotent $e \in  \langle G,t\rangle$, and thus $tg^{-1}=e$. Because $e$ and $t$ have the same kernel classes and every point in the image of $e$ is fixed, we have that $A_i e \subseteq A_i$ for all $i$. It follows that    $\{a_i g^{-1}\}=A_itg^{-1}=A_i e\subseteq  A_i$ (for all $i \in \{1,\ldots, k\}$). The result follows. 
\qed  
\end{proof}

Glueing together the previous result with the classification of the groups possessing the  ordered $k$-ut in Propositions \ref{p:ord6ut}, \ref{p:ord5ut}, \ref{p:ord4ut}, and \ref{p:ord3ut}, we get Theorems \ref{rank2}--\ref{rank3} which are the main theorems of the first part of this paper. 

\section{Semigroups and their normalizers}\label{four}

Let $S\le \trans$ be a semigroup and let $G\le \sym$ be its normalizer in $\sym$. We are interested in the relation between $S$ and  $\langle S,G\rangle$. On one hand the semigroup $\langle S,G\rangle$  might be more accessible to study since we can take advantage of group theoric results, but on the other hand the properties of $S$ might be very different from the properties of  $\langle S,G\rangle$. 

For example, we might be unable to verify if a given semigroup  $S$ is regular.  If all $t\in S$ have rank at most $k$ and if $G$  has the $k$-ut property, then the semigroup $\langle S,G\rangle$ is easily seen to be regular. Hence, to prove the regularity of $S$, we need to prove that  regularity of $\langle S,G\rangle$  implies regularity of $S$.  Therefore, the goal of this section is to study semigroup properties that carry from $\langle S,G\rangle$ to $S$. 

We start by proving a general result. 
\begin{lemma}
Let $S\le \trans$ and let $G$ be its normalizer in $\sym$. Then 
\[
\langle S,G\rangle = SG. 
\]  
\end{lemma}

\begin{proof}
For $s\in S$ and $g\in G$ let $s^g$ denote $g^{-1}sg$. Let $t\in \langle S,G\rangle$. We now have (for some $g_1, \dots, g_{k+1} \in G$, $s_1, \dots s_k \in S$),
\[\begin{array}{rcl}
t&=&g_1s_1g_2s_2\ldots g_ks_kg_{k+1}\\
  &=&g_1s_1g^{-1}_1(g_1g_2)s_2(g_1g_2)^{-1}(g_1g_2g_3)\ldots (g_1\ldots g_k)s_k(g_1\ldots g_k)^{-1}(g_1\ldots g_kg_{k+1})\\
  &=&s^{g_1}_1s^{g_1g_2}_2\ldots s^{g_1\ldots g_k}_k (g_1\ldots g_{k+1})\\
  &=&sg \in SG,
\end{array}
\]where $s=s^{g_1}_1s^{g_1g_2}_2\ldots s^{g_1\ldots g_k}_k\in S$ and $g=g_1\ldots g_{k+1}\in G$. Thus $\langle S,G\rangle\subseteq SG$. The reverse inclusion is obvious. 
\qed 
\end{proof}

\subsection{Regularity}
 Recall that a semigroup $S$  is regular if for all $a\in S$ there exists $a'\in S$ such that $a=aa'a$. Two elements $a,b\in S$ are said to be ${\mathcal R}$-related if there exist $u,v\in S^1$ such that $a=bu$ and $b=av$ ($S^1$ denotes the monoid obtained by adjoining an identity to $S$). Similarly, $a,b\in S$ are said to be ${\mathcal L}$-related if there exist $u,v\in S^1$ such that $a=ub$ and $b=va$. It is well know that a semigroup is regular if and only if  every element is ${\mathcal R}$-related (or ${\mathcal L}$-related) to an idempotent. In what follows, by a transformation monoid $S\le \trans$ we mean a semigroup of transformations containing the identity transformation. 
  
The key result in this subsection is the following lemma. 

\begin{lemma}\label{main}
Let $S\le T_n$ be a transformation monoid  and $G$ be the normalizer of $S$ in the symmetric group. If $a\in S$ is ${\mathcal R}$-related in $SG$ to an idempotent of $SG$, then $a$ is ${\mathcal R}$-related in $S$ to the same  idempotent. 
\end{lemma}
\begin{proof}
Let $a\in S$ and assume that $a$ is ${\mathcal R}$-related in $SG$ to an idempotent in $SG$, that is, there exist $b\in S, h \in G$ such that $(bh)(bh)=bh$ and for some $b_1i_1,b_2i_2\in SG$, we have $a=(bh)(b_1i_1)$ and $bh=a(b_2i_2)$.  The claim is clearly true if $bh$ is the identity, so assume this is not case.

Now, by a theorem of McAlister, for all $s\in S$, the semigroups $\langle s,G\rangle \setminus G $ and $\langle g^{-1}sg\mid g\in G\rangle\setminus G$  have the same idempotents (\cite[Lemma 2.2]{mcalister} and \cite[Lemma 2.2]{AMS}). As $S=\bigcup_{s\in S}\langle g^{-1}sg\mid g\in G\rangle$ and $SG=\bigcup_{s\in S}\langle s,G\rangle$, it follows that $S$ and $SG$ have the same idempotents. Thus $bh\in S$. 

It remains to prove that $a$ and $bh$ are ${\mathcal R}$-related in $S$, that is, there exist $u,v\in S$ such that $a=(bh)u$ and $bh=av$. Since $(bh)(b_1i_1)=a\in S$, we can take $u= (bh)(b_1i_1)$ so  that $a=(bh)u=(bh)\ (bh)(b_1i_1)$. 

Observe that $a(b_2 i_2)(bh)=(bh)(bh)=bh$. We claim that $i_2bh\in S$ and hence $b_2i_2bh\in S$, thus proving the theorem. 

We start by proving that $hbh\in S$. In fact, 
\[\begin{array}{rcl}
h^{-1}bh,b\in S&\Rightarrow&  h^{-1}bhb\in S\Rightarrow  h^{-2}bhbh=h^{-2}bh\in S\Rightarrow h^{-2}bhb\in S\Rightarrow \\
                        &\Rightarrow& h^{-3}bhbh=h^{-3}bh\in S\Rightarrow \ldots \Rightarrow  h^{-k}bh\in S.    
\end{array}\]
As $G$ is finite, for some $k$ we have $h^{-k}=h$. The claim follows.

Now we claim that $i^{-k}_2b\in S$ for all natural numbers $k$. We proceed by induction. From  $a,b_2\in S$ we get $ab_2\in S$ and hence $i^{-1}_2ab_2i_2\in S$, thus $ i^{-1}_2 bh \in S$ so that $ i^{-1}_2 bhb \in S$; as $bhbh=bh$, we have $bhb=b$ which together with  $ i^{-1}_2 bhb \in S$ yields $ i^{-1}_2 b \in S$. 

Now suppose that $i^{-k}_2 b\in S$ (for some natural $k\ge 1$); we want to prove that $i^{-(k+1)}_2b\in S$. From $i^{-(k+1)}_2b_2i^{k+1}_2,i^{-k}_2b\in S$, we get  $i^{-(k+1)}_2b_2i^{(k+1)}_2i^{-k}_2b=i^{-(k+1)}_2b_2 i_2 b\in S$. Thus 
\[\begin{array}{rcccl}
S&\ni& \left( i^{-(k+1)}_2ai^{(k+1)}_2        \right)   \left(    i^{-(k+1)}_2b_2 i_2 b           \right)&=& i^{-(k+1)}_2ai^{(k+1)}_2 i^{-(k+1)}_2b_2 i_2 b\\  
   &    &                                                                                                                                          &=&i^{-(k+1)}_2ab_2 i_2 b\\
   &     &                                                                                                                                         &=& i^{-(k+1)}_2bh b \\
   &	   &																	&=&i^{-(k+1)}_2b. 
\end{array}
\] It is proved that $i^{-k}_2b\in S$ for all natural $k$. 

As $G$ is finite, for some $k$ we have $i_2b=i^{-k}_2b\in S$. Since $i_2b,hbh\in S$, it follows that $i_2bhbh=i_2bh\in S$. As $b_2,i_2bh\in S$, we get 
 $b_2i_2bh\in S$ and hence $a(b_2i_2bh)=(ab_2i_2)(bh)=(bh)^2=bh$. It is proved that $a$ and the idempotent $bh$ are ${\mathcal R}$-related in $S$. \qed
\end{proof}

By symmetry we get the following.

\begin{lemma}
Let $S\le T_n$ be a transformation monoid and $G$ be the normalizer of $S$ in the symmetric group $G$. If $a\in S$ is ${\mathcal L}$-related in $SG$ to an idempotent of $SG$, then $a$ is ${\mathcal L}$-related in $S$ to the same idempotent. 
\end{lemma}

Two elements $a,b \in S$ are said to be ${\mathcal H}$-related if they are ${\mathcal R}$-related and ${\mathcal L}$-related.
The two previous results imply the following. 

\begin{lemma}\label{hmain}
Let $S\le T_n$ be a transformation monoid and $G$ be the normalizer of $S$ in the symmetric group $G$. If $a\in S$ is ${\mathcal H}$-related in $SG$ to an idempotent of $SG$, then $a$ is ${\mathcal H}$-related in $S$ to the same idempotent. 
\end{lemma}

A number of consequences follow from these lemmas. 

\begin{cor}
Let $S\le T_n$ be a transformation monoid and $G$ be the normalizer of $S$ in the symmetric group $\sym$. 
Then  $SG$ is regular if and only if  $S$ is regular.
\end{cor}

\begin{proof}
Let $a\in S$. As $SG$ is regular and $S\le SG$, it follows that $a$ is ${\mathcal R}$-related in $SG$ to an idempotent of $SG$. By Lemma \ref{main},  $a$ is ${\mathcal R}$-related in $S$ to the same idempotent (which is thus in $S$). We conclude that every element in $S$ is ${\mathcal R}$-related to an idempotent. 

Regarding the converse, suppose $S$ is regular, say $s=ss's$ and $s'=s'ss'$, for all $s\in S$. Let $sg\in SG$. Then $sg=(sg)(g^{-1}s')(sg)$ and $g^{-1}s'=(g^{-1}s')(sg)(g^{-1}s')$; in addition $g^{-1}s'=(g^{-1}s'g)g^{-1}\in SG$. It is proved that every element in $SG$ has an inverse in $SG$. The result follows.
\qed
\end{proof}

We observe that it is possible for non-idempotents $p,q\in S$ to be ${\mathcal R}$-related in $SG$, but not in $S$. For example, pick $g,t,q\in T_7$ as follows:  $g:=(567)$, 
\[\begin{array}{ccc}
t:=\left( \begin{array} { c c c c } { { 1,2,3,4 } } & { 5 } & { 6 } & { 7 } \\ { 1 } & { 2 } & { 3 } & { 4 } \end{array} \right)&\mbox{ and }&
q:=\left( \begin{array} { c c c c } { { 1,2,3,4 } } & { 5 } & { 6 } & { 7 } \\ { 1 } & { 3 } & { 4 } & { 2 } \end{array} \right)
\end{array}
\]

Then $S:=\langle g,t,q\rangle$ has $7$ elements and its normalizer $G$ in $S_7$ is generated by $\langle g, (34)(76),(23)(76)\rangle$. We have $t=q(243)$ and $q=t(234)$, thus $t$ and $q$ are ${\mathcal R}$-related in $SG$, but they are not ${\mathcal R}$-related in $S$.

Recall that a semigroup is completely regular if each element belongs to a maximal subgroup; equivalently, every ${\mathcal H}$-class contains an idempotent. Therefore, Lemma \ref{hmain} implies the following.

\begin{cor}
Let $S\le T_n$ be a transformation monoid and $G$ be the normalizer of $S$ in the symmetric group $\sym$. Then $SG$ is completely regular if and only if  $S$ is completely regular. 
\end{cor}

An element $a$ of a semigroup $S$ is said to be {\em abundant} if it is ${\mathcal R}$-related and ${\mathcal L}$-related to an idempotent in some oversemigroup $T$ of $S$.  Therefore, if an element $a\in S$ is abundant in $SG$, this means that $a$ is also abundant in $S$. The {\em abundant world} generalizes the regular world, but in this context adds nothing. 

An element $a$ of a semigroup $S$ is said to be {\em right [left] inverse} if it is ${\mathcal R}$-related [${\mathcal L}$-related] to exactly one idempotent in $S$; the element is inverse if it is both left and right inverse.  A semigroup is inverse if all of its elements are inverse. As seen above, if $a\in S$ is ${\mathcal R}$-related to exactly one idempotent in $SG$, by Lemma \ref{main} we know that in $S$ the element $a$ is ${\mathcal R}$-related to the same idempotent (that belongs to $S$); thus, if $a\in S$ is right inverse in $SG$, it is also right inverse in $S$;  as a consequence, if  $SG$ is an inverse semigroup, then so is $S$. (This last conclusion follows immediately from the fact  if $SG$ is inverse, it is regular and the idempotents commute and hence the same is true in $S$.)  

A semigroup is said to be \emph{Clifford} if it is inverse and completely regular. By the results above it follows that if $SG$ is Clifford, then so is $S$. 

A monoid $S$ is {\em intra-regular} if for all $a\in S$ there exist $b,c \in S$ such that $ba^2c= a$. 
We would like to know if $SG$ intra-regular implies $S$ intra-regular, but we only have the following partial result. 

\begin{prop}
Let $G\le \sym$ be a group of exponent~$2$ and let $S\le \trans$ be a transformation monoid. Then $SG$ is intra-regular implies that $S$ is intra-regular. 
\end{prop}
\begin{proof}
As $S$ is a finite semigroup, for each $x\in S$ there exist natural numbers $l$ and $m$, $m>l$, such that $x^l=x^m$. 

Let $a \in S$ be arbitrary, and $l$ and $m$ as above. 
If $l=1$ and $m=2$, then  $a=a^2$ and hence $a=1aa1$, with $1\in S$;  if $l=1$ and $m>2$, then $a=1a^2a^{m-2}$, so the result holds for all $a$ for which $l=1$. 

Assume instead that $l,m>1$. By  intra-regularity in $SG$, there exist $eg,fh\in SG$ such that $a=ega^2fh$, with $e,f\in S$ and $g,h\in G$. It is clear that $f,hfh^{-1}\in S$ and hence $fhfh^{-1}\in S$; as $h^{-1}=h$ it follows that $fhfh\in S$. We claim that $egaeg\in S$. In fact $ae\in S$ because $a,e\in S$; thus $gaeg^{-1},e\in S$ and hence $egaeg^{-1}=egaeg\in S$.  The claim follows. Now
\[
(egaeg)a^2(fhfh)=ega(ega^2fh)fh=ega(a)fh=ega^2fh=a. 
\]It is proved that every intraregular element in $SG$ is intraregular in $S$, when $G$ satisfies $x^2=1$. \qed 
\end{proof}



\subsection{Semigroups having square roots}

The aim of this subsection is to carry the foregoing investigation to the case of $SG$ having square roots. Observe that every finite semigroup satisfies an identity of the form $x^m=x^k$, for $0\le k<m$. If $a,b\in S$ and $b^2=a$ we say that $b$ is a square root of $a$; we denote an arbitrary square root of $x$ by $\sqrt{x}$ and the notation $\sqrt{x}\in A\subseteq S$ means that   $x$ has a square root in the set $A$. 

\begin{theorem}\label{mainsquare}
Let $S$ be a finite semigroup in which every element has a square root. Then for every $x\in S$ we have $\sqrt{x}\in \langle x \rangle$. 
\end{theorem}
\begin{proof}
Let $s: S\to S$  given by $s(x)=x^2$. That $S$ has square roots for all of its elements means that $s$ is surjective, and as $S$ is finite, it is also injective. But for any $x \in S$,
$s(\langle x \rangle) \subseteq  \langle x \rangle$. Let $t$ be the restriction of $s$ to $ \langle x \rangle$. Then $t$ is injective (because $s$ is), and by finiteness, also surjective. 
Hence $x$ has a square root in $ \langle x \rangle$.\qed
\end{proof}

Now back to the leitmotiv of this paper. 

\begin{cor}
Let $S\le T_n$ be a transformation monoid and $G$ be the normalizer of $S$ in the symmetric group $G$. If every element in  $SG$ has a square root, then so has every element in  $S$.  
\end{cor}
\begin{proof}
As $SG$ has square roots for all its elements, it follows that for every $x\in SG$ we have $\sqrt{x}\in \langle x\rangle$ and hence every $x\in S$ has square root in $S$.  
\qed\end{proof}

It is well known that every element of a finite group has a square root if and only if the group  has odd order. The next theorem extends this result to semigroups. 

\begin{cor}
Every element of a  finite semigroup $S$ has a square root if and only if every $s\in S$ belongs to an odd order maximal subgroup of $S$. 
\end{cor}
\begin{proof}
Let every element of $S$ have a square root, and $x\in S$. By Theorem \ref{mainsquare}, $\sqrt{x}\in \langle x\rangle$, that is, $\sqrt{x}=x^n$, for some natural $n$. Therefore, $x=x^{2n}$ and hence $\langle x\rangle$ is a cyclic group with $x^{-1}=x^{2n-2}$ and  $x^{2n-1}$ being the identity element. This means that the ${\mathcal H}$-class of $x$ is a group (since it has an idempotent); as every element in this group has a square root, which necessarily lies in it, this group must have odd order. 

Conversely, if $S$ is a union of odd order groups, then  every element in $S$ has a square root.   \qed
\end{proof}

It is well known that neither the symmetric group nor the full transformation monoid contain square roots of all their elements, except in the trivial cases. With our main theorem at hand we can say a bit more: if $S$ is a finite semigroup and $a\in S$ has no square root, then it is not possible to extend $S$ to a finite oversemigroup $T$ containing a square root for $a$ (because a square root of $a$ must belong to $\langle a\rangle$ and this semigroup remains the same in any oversemigroup of $S$).

\subsection{A negative result}

A semigroup $S$ is said to be {\em ${\mathcal R}$-commutative} if for all  $a,b\in S$ we have $ab{\mathcal R}ba$. Let $S<T(\{1,\ldots,7\})$ be the semigroup generated by the permutations $(24)(36)$, $(15)(23)(46)$ and the transformation 
\[
t =\left(
\begin{array}{ccc}
\{1,5,7\}&\{2,3\}&\{4,6\}\\
7&1&5
\end{array}
\right).
\]
The normalizer $G$ of $S$ in the symmetric group is the group generated by $(15)$, $(24)(36)$ and $(15)(23)(46)$. GAP shows that the semigroup $SG$ is ${\mathcal R}$-commutative, but $S$ is not. By symmetry, a corresponding counterexample exists for ${\mathcal L}$-commutativity.

\section{Problems}\label{five}
We now propose a number of problems. The first  is essentially in \cite{ac} but (annoyingly) keeps resisting. 
\begin{problem}
Complete  the classification of the groups possessing the  $3$- and $4$-ut property so that Theorem \ref{rank4} and Theorem \ref{rank3} can be completed.  
\end{problem}


%


There is a well known correspondence between the behaviour of $T(\Omega)$ and $\End(V)$, when $\Omega$ is a finite set and $V$ is a finite dimension vector space. This prompted the introduction of {\em independence algebras}, a class containing both sets and vector spaces as particular cases. Therefore the next two problems turn out to be very natural. 

\begin{problem}
Prove linear analogues of the main theorems in this paper. 
\end{problem}

\begin{problem}
Find in the context of independence algebras analogues of the main theorems in this paper. 
\end{problem}


Let $S$ be a finite semigroup. Two elements $a,b\in S$ are said to be $\mathcal{J}$-related if they generate the same principal ideal, that is, $S^1aS^1=S^1bS^1$.  We recall (\cite[p.57]{Ho95}) that if $S<T$ is a regular subsemigroup of a semigroup $T$, then ${\mathcal R}_S={\mathcal R}_T\cap (S\times S)$. The same happens for ${\mathcal L}$ or ${\mathcal H}$, but fails for ${\mathcal J}$. 

\begin{problem}
Let $S\le T_n$ and $G\le \sym$ be its normalizer in $\sym$. Is it true that  ${\mathcal J}_S={\mathcal J}_{SG}\cap (S\times S)$?
\end{problem}
     
     An {\em existential property} of semigroups is a first order language condition on the elements of the semigroup that uses an existential quantifier. For example, regularity is an existential property of semigroups. 
 \begin{problem}
 Let $S\le \trans$ be a semigroup and $G$ its normalizer in $\sym$. Let $P$ be an existential property of semigroups. Decide if $SG$ satisfies $P$ implies that $S$ also satisfies $P$. 
\end{problem}

The following result was proved for groups satisfying $g^2=1$. Can it be generalized for other classes of groups?

\begin{problem}
Let $G\le \sym$ be a group and let $S\le \trans$ be a transformation monoid. Is it true that if  $SG$ is intra-regular then  $S$ is intra-regular? 
\end{problem}

We close the list of problems with a general semigroup structure question.

\begin{problem}
Consider $G$, one of the groups appearing in Theorems \ref{rank2}--\ref{rank3}. Describe the structure (Green's relations, automorphisms, congruences, conjugacy classes, the variety generated, etc.) of the semigroups $\langle G,a\rangle$, for $a\in \trans$.  
\end{problem}

\end{document}